\documentclass[12pt]{article}

\usepackage{amsmath}
\usepackage[colorlinks=true]{hyperref}
\usepackage{amssymb}
\usepackage{amsthm}
\usepackage{latexsym}
\usepackage{cite}
\usepackage{psfrag}
\usepackage{color}
\usepackage{amsfonts}
\usepackage{mathrsfs}
\usepackage{url}
\hypersetup{urlcolor=blue, citecolor=red}
\usepackage{psfrag}
\usepackage{epsfig}
\usepackage{graphicx}

\makeatletter
\@addtoreset{equation}{section}
\makeatother

\newtheorem{theorem}{Theorem}[section]
\newtheorem{corollary}[theorem]{Corollary}
\newtheorem{proposition}[theorem]{Proposition}
\newtheorem{lemma}[theorem]{Lemma}
\theoremstyle{definition}

\newtheorem{example}[theorem]{Example}

\newcommand{\As}{\mathscr{A}}
\newcommand{\U}{\mathscr{U}}

\newcommand{\A}{\mathcal{A}}
\newcommand{\K}{\mathcal{K}}
\newcommand{\N}{\mathcal{N}}
\newcommand{\Q}{\mathcal{Q}}
\newcommand{\V}{\mathcal{V}}
\newcommand{\W}{\mathcal{W}}

\newcommand{\PG}{\mathrm{PG}}
\newcommand{\F}{\mathbb{F}}

\def\db{\displaybreak[3]}
\def\dbn{\displaybreak[3]\notag}

\begin{document}
\title
{New bounds for covering codes of radius 3 and codimension $3t+1$
\date{}
\thanks{The research of S. Marcugini and F. Pambianco was supported in part by the Italian
National Group for Algebraic and Geometric Structures and their Applications (GNSAGA -
INDAM) (Contract No. U-UFMBAZ-2019-000160, 11.02.2019) and by University of Perugia
(Project No. 98751: Strutture Geometriche, Combinatoria e loro Applicazioni, Base Research
Fund 2017-2019).}
}
\maketitle
\begin{center}
{\sc Alexander A. Davydov}\\
 \emph{E-mail address:} alexander.davydov121@gmail.com\medskip\\
 {\sc Stefano Marcugini and
 Fernanda Pambianco }\\
 {\sc\small Department of  Mathematics  and Computer Science,  Perugia University,}\\
 {\sc\small Perugia, 06123, Italy}\\
 \emph{E-mail address:} \{stefano.marcugini, fernanda.pambianco\}@unipg.it
\end{center}

\textbf{Abstract.}
The smallest possible length of a $q$-ary linear code of covering radius $R$ and codimension (redundancy) $r$ is called the length function and is denoted by $\ell_q(r,R)$. In this work, for $q$ \emph{an arbitrary prime power}, we obtain the following new constructive  upper bounds on $\ell_q(3t+1,3)$:
\begin{equation*}
\begin{split}
    &\bullet~\ell_q(r,3)\lessapprox \sqrt[3]{k}\cdot q^{(r-3)/3}\cdot\sqrt[3]{\ln q},~r=3t+1, ~t\ge1,~  q\ge\lceil\W(k)\rceil,\\
 &\phantom{\bullet~}18 <k\le20.339,~\W(k)\text{ is a decreasing function of }k ;\\
 &\bullet~\ell_q(r,3)\lessapprox \sqrt[3]{18}\cdot q^{(r-3)/3}\cdot\sqrt[3]{\ln q},~r=3t+1,~t\ge1,~ q\text{ large enough}.
  \end{split}
\end{equation*}
For $t = 1$, we use a one-to-one correspondence between codes of covering radius 3 and codimension 4,  and 2-saturating sets in the projective space $\mathrm{PG}(3,q)$. A new
construction providing sets of small size is proposed. The codes, obtained by geometrical methods, are taken as the starting ones in the lift-constructions (so-called ``$q^m$-concatenating constructions'')  to obtain infinite families of codes with radius 3 and growing codimension $r = 3t + 1$, $t\ge1$.
The new bounds are essentially better than the known ones.



\section{Introduction}\label{sec:Intr}
Let $\F_{q}$ be the Galois field with $q$ elements. Let $\F_{q}^*=\F_{q}\setminus{0}$. Let $\F_{q}^{\,n}$ be the $n$-dimensional vector space over~$\F_{q}.$ A linear code in $\F_{q}^{\,n}$
 with codimension (redundancy) $r$, minimum distance $d$, and
covering radius $R$ is said to be an $[n, n- r,d]_q R$ code; $d$ is omitted when not relevant. All words in $\F_{q}^{\,r}$
 can be obtained as a linear combination of at most $R$ columns of a parity
check matrix of an $[n, n-r,d ]_q R$ code. Also, for this code, the space  $\F_{q}^{\,n}$ is covered by spheres of radius $R$ centered at the codewords.
For an introduction to coding theory, see \cite{HufPless,MWS}.

The minimum length $n$ such that an $[n, n-r ]_q R$ code exists is called \emph{the length function} and is denoted by $\ell_q(r,R)$. If covering radius and
codimension are fixed then the covering problem for codes is that of finding codes with small length. Codes investigated from the point
of view of the covering problem are called \emph{covering codes}. Studying covering codes is a classical combinatorial problem. Covering codes
 are connected with many areas of theory and practice, see e.g. \cite[Section 1.2]{CHLS-bookCovCod}, \cite[Introduction]{DMP-AMC2021}, \cite{Graismer-2023}. For an introduction to coverings of  Hamming spaces over finite fields and covering
codes, see \cite{Handbook-coverings,CHLS-bookCovCod,DGMP-AMC,Struik,LobstBibl}.

This paper is devoted to the upper bound on the length function $\ell_q(3t+1,3)$, $t\ge1$.

Let $\PG(N,q)$ be the $N$-dimensional projective space over the Galois field $\mathbb{F}_q$. To obtain bounds on $\ell_q(4,3)$ we use 2-\emph{saturating sets in $\PG(3,q)$}.
A point set $S\subset\PG(3,q)$ is
2-saturating if
 any  point of $\PG(3,q)$ belongs to a plane generated by three non-collinear points of $S$.
Let $s_q(3,2)$ be the smallest size of a 2-saturating set in $\PG(3,q)$.

If the positions of a column of a parity check matrix of an $[n,n-4]_q3$ code are treated as homogeneous coordinates of a point in $\PG(3,q)$ then the matrix defines a 2-saturating $n$-set  in $\PG(3,q)$, and vice versa. So, there is a \emph{one-to-one correspondence} between $[n,n-4]_q3$ codes and 2-saturating $n$-sets in $\PG(3,q)$. Therefore,
$\ell_q(4,3)=s_q(3,2).$

For an introduction to geometries over finite fields and their connections with coding theory, see \cite{EtzStorm2016,Giul2013Survey,Hirs,Hirs_PG3q,HirsStor-2001,LandSt,DGMP-AMC}.

Throughout the paper, $c$ is a constant  independent of $q$ but it is possible that $c$ is dependent on $r$ and $R$.
In \cite{BDGMP-R2R3CC_2019,DMP-R=3Redun2019}, \cite[Proposition 4.2.1]{denaux}, see also the references therein, the following lower bound is considered:
\begin{equation}\label{eq1:lowbound}
\ell_q(r,R)\ge cq^{(r-R)/R},~  R\text{ and }r\text{ fixed}.
\end{equation}
In \cite{DGMP-AMC}, the bound \eqref{eq1:lowbound} is given in another \mbox{(asymptotic)} form.
Let $t,s$ be integer. Let $q'$ be a prime power. In the literature, the bound \eqref{eq1:lowbound} is achieved in the following cases:
 \begin{align*}
 &r\ne tR,~  q=(q')^R~ \text{\cite{DGMP-AMC,denaux,denaux2,HegNagy}};
 R=sR',~ r=tR+s, ~q=(q')^{R'}~\text{\cite{DGMP-AMC,DMP-R=tR2019}};\\
 &r=tR,~q\text{ is an arbitrary prime power}~ \text{\cite{Dav95,DavOst-IEEE2001,DGMP-AMC,DMP-R=tR2019,DavOst-DESI2010}}.
  \end{align*}

In the general case, for arbitrary $r,R, q$, in particular when $r\ne tR$ and $q$ is an arbitrary prime power, the problem
of achieving the bound \eqref{eq1:lowbound} is open.

For $r=tR+1$, in \cite{BDGMP-R2Redun2016,BDGMP-R2Castle,BDGMP-R2R3CC_2019,BDMP-arXivR3_2018,DMP-R=3Redun2019,DMP-ICCSA2020,Bo-Sz-Ti,Nagy} (for $R=2,3$)
and \cite{DMP-AMC2021} (for $R\ge3$) upper bounds of the following form are obtained:
\begin{equation}\label{eq1:upbound}
\ell_q(tR+1,R)\le cq^{(r-R)/R}\cdot\sqrt[R]{\ln q},~ q\text{ is an arbitrary prime power}.
  \end{equation}
In the bounds \eqref{eq1:upbound}, the ``price'' of the non-restrict structure of $q$ is the factor $\sqrt[R]{\ln q}$.

In this paper, see Sections \ref{subsec:explicit} and \ref{subsec:BoundD}, for $q$ \emph{an arbitrary prime power}, we obtain the following new constructive  upper bounds on $\ell_q(3t+1,3)$ of the form \eqref{eq1:upbound}:
\begin{equation*}
 \begin{split}
& \bullet~\ell_q(r,3)\lessapprox \sqrt[3]{k}\cdot q^{(r-3)/3}\cdot\sqrt[3]{\ln q},~r=3t+1,~ t\ge1, ~ q\ge\lceil\W(k)\rceil,\\
&\phantom{\bullet~}18 <k\le20.339,~\W(k) \text{ is a decreasing function of }k;\\
 &\bullet~\ell_q(r,3)\lessapprox \sqrt[3]{18}\cdot q^{(r-3)/3}\cdot\sqrt[3]{\ln q},~r=3t+1,~t\ge1,~ q\text{ large enough}.
 \end{split}
  \end{equation*}

 We consider the case $t=1$ in projective geometry language. We propose a step-by-step algorithm
 obtaining a 2-saturating $n$-set  in $\PG(3,q)$ that is a subset of an elliptic quadric and corresponds to an $[n,n-4,4]_q3$ code, see Section \ref{sec:2sat constr}.
Estimates of the size of the obtained set give the bounds on $\ell_q(4,3)$, see Section \ref{sec:2sat}.

For $t\ge2$, we use the lift-constructions for covering codes.  These constructions are variants of the so-called ``$q^m$-concatena\-ting constructions'' proposed in \cite{Dav90PIT} and developed in \cite{Dav95,DGMP-AMC,DavOst-IEEE2001,DMP-IEEE2004,DavOst-DESI2010,DMP-AMC2021,DMP-ICCSA2020}, see also the references therein and \cite{Handbook-coverings}, \cite[Section~5.4]{CHLS-bookCovCod}. The $q^m$-concatenating constructions obtain infinite families of covering codes with growing codimension using a starting code with a small one.
We take the $[n,n-4,4]_q3$ codes with $t=1$ as the starting ones for the $q^m$-concatenating constructions and obtain infinite families of covering codes with growing codimension $r=3t+1$, $t\ge1$. These families provide the constructive upper bounds on $\ell_q(3t+1,3)$, $t\ge1$.

  The new bounds are essentially better than the best known ones of \cite{DMP-AMC2021}; for details
see the figures and the table in Section \ref{sec:2sat} and the figure and the relations in Section \ref{sec:compare}. In particular, in the region $14983\le q<5\cdot10^6$, the ratio of values of the known and new upper bounds lies in the region $2.167\ldots1.96$ and the asymptotic ratio is $\thickapprox1.89$.

The 2-saturating $n$-set  in $\PG(3,q)$, obtained in Section \ref{sec:2sat constr} of this work, corresponds to an $[n,n-4,4]_q3$ code with $r=4$ and $d=4$, whereas the known bound  for $r=4$ is provided by a code with $d=3$.

Note also that, for small $q$ in the region $7577 < q \le 7949$, we obtain by computer search new small 2-saturating sets in $\PG(3,q)$, see Section \ref{subsecComput}.

The paper is organized as follows. In Section \ref{sec:2sat constr}, a construction of 2-saturating sets in $\PG(3,q)$ is proposed. In Section \ref{sec:2sat}, estimates of sizes of 2-saturating sets in $\PG(3,q)$, obtained by the proposed construction, and new bounds on $\ell_q(4,3)$ are considered. In Section \ref{sec:r=3t+1}, with the help of the $q^m$-concatenating lift-constructions the new bounds on $\ell_q(3t+1,3)$, $t\ge1$, are obtained. Finally, in Section \ref{sec:compare}, the best known upper bounds of \cite{DMP-AMC2021} are given and are compared with the new ones.

\section{A construction of 2-saturating sets in $\PG(3,q)$}\label{sec:2sat constr}
We say that
a point $P$ of $\PG(3,q)$ is \emph{covered} by a point set $\K\subset\PG(3,q)$ if $P$ lies in a plane through 3 non-collinear  points of $\K$ or on a line through 2 points of $\K$.  In these cases, we say also that the set $\K$ \emph{covers} the point $P$.
Obviously, all points of $\K$ are covered by $\K$.

Let $\Q\subset\PG(3,q)$ be an elliptic quadric \cite{Hirs_PG3q}; then $\#\Q=q^2+1$. Let $B_u\in\PG(3,q)$ be a point of $\Q$. So, $\Q=\{B_1,B_2,\ldots,B_{q^2+1}\}.$
We construct a 2-saturating set as a \emph{subset of the elliptic quadric} by a step-by-step iterative process
which adds a new point to the current set in every step.

Let $w>0$ be a fixed integer. Consider the $(w+1)$-st step of the process.
This step  starts from a $w$-subset $\K_w=\{B_1,B_2,\ldots,B_w\}\subset\Q$ obtained in the previous $w$ steps.
Let $B_{w+1}\in \Q\setminus \K_{w}$ be the point that will be added to the subset in the $(w+1)$-st step, i.e. $\K_{w+1}=K_{w}\cup \{B_{w+1}\}$.
Let $\U_w$ and $\U_{w+1}$ be the subset of points of $PG(3,q)$ that are not covered by $\K_{w}$ and $\K_{w}\cup \{B_{w+1}\}$, respectively.
Clearly, $\U_{w+1}\subseteq\U_w$.

To obtain the next subset
$\K_{w+1}$, we can add to $\K_{w}$ any of $q^2+1-w$ points of $\Q\setminus \K_{w}$.
Thus, there exist $q^2+1-w$ distinct points that can be taken as $B_{w+1}$.

Let $\Delta(H)$ be the number of the \emph{new covered points} of $PG(3,q)$ in the $(w+1)$-st step after adding a point $H\in \Q\setminus \K_{w}$ to $\K_w$.
Let $\N(H)$ be the set of new points  of $PG(3,q)$ covered by $\K_{w}\cup \{H\}$ with $H\in \Q\setminus \K_{w}$. Obviously, $\#\N(H)=\Delta(H)$.
 Introduce the point multiset $\N_{w+1}^{\,\cup}$ such that
 \begin{equation*}
   \N_{w+1}^{\,\cup}\triangleq\bigcup\limits_{H\in \Q\setminus \K_{w}}\N(H),~\#\N^{\,\cup}_{w+1}=\sum\limits_{H\in \Q\setminus \K_{w}}\Delta(H).
 \end{equation*}

 Let $P\in \U_w$ be a point  that is not covered by $\K_{w}$.  Let $S_w(P)$ be the number of the inclusions of $P$ to the multiset $\N_{w+1}^{\,\cup} $ when one sequentially adds all points $H$ of $\Q\setminus\K_w$ to $\K_w$ to obtain variants $\K_{w+1}=\K_w\cup\{H\}$. Let $S_w^{\min}\triangleq\min\limits_{P\in \U_w}S_w(P)$. Then
 \begin{equation} \label{eq2:NcupSwP}
   \#\N^{\,\cup}_{w+1}=\sum_{P\in \U_w}S_w(P)\ge S_w^{\min}\cdot\#\U_w.
 \end{equation}

Let $\Delta _{w+1}^{\text{aver}}$ be the average number of new covered points in the  $(w+1)$-st step calculated over all $\#\Q\setminus\K_w= q^2+1-w$ possible  points $H$. By \eqref{eq2:NcupSwP},
\begin{equation}\label{eq2:aver_def}
    \Delta _{w+1}^{\text{aver}}=\frac{ \#\N^{\,\cup}_{w+1}}{q^2+1-w}\ge\frac{ S_w^{\min}\cdot\#\U_w}{q^2+1-w}.
\end{equation}
Obviously, there exists a point $H^*\in\Q\setminus\K_w$ providing the inequality $\Delta(H^*)\ge\left\lceil\Delta _{w+1}^{\text{aver}}\right\rceil$. \emph{We take this point $H^*$ as $B_{w+1}$.} Then, by \eqref{eq2:aver_def},
\begin{equation}\label{eq2:>=aver}
   \Delta(B_{w+1})=\#\U_w-\#\U_{w+1}\ge\left\lceil\Delta _{w+1}^{\text{aver}}\right\rceil\ge\left\lceil\frac{ S_w^{\min}\cdot\#\U_w}{q^2+1-w}\right\rceil\ge\frac{ S_w^{\min}\cdot\#\U_w}{q^2+1-w};
   \end{equation}
\begin{equation}
  \#\U_{w+1}\le \#\U_w-\left\lceil\frac{ S_w^{\min}\cdot\#\U_w}{q^2+1-w}\right\rceil\le\#\U_w\left(1-\frac{ S_w^{\min}}{q^2+1-w}\right).\label{eq2:>=aver2}
\end{equation}

The  iterative process ends when $\#\U_{w+1}\le 1$. Then one should add one point of $\Q$ to $\K_w$ to obtain a 2-saturating set.

\begin{theorem}\label{the2:d=4}
  The $2$-saturating $n$-set obtained by the construction above corresponds to an $[n,n-4,4]_q3$ code with minimum distance $d=4$.
\end{theorem}

\begin{proof}
  All the points of the $2$-saturating $n$-set belong to the elliptic quadric $\Q$.
\end{proof}
\section{Estimates of sizes of 2-saturating sets in $\PG(3,q)$ and new bounds on $\ell_q(4,3)$}\label{sec:2sat}
 Let $\theta_{N,q}=(q^{N+1}-1)/(q-1)$ be the number of points in the projective space $\PG(N,q)$.
Obviously, any three points of $\Q$ cover some plane; therefore
\begin{equation}\label{eq3:U3}
  \#\U_3=\theta_{3,q}-\theta_{2,q}=q^3.
\end{equation}

We consider our new bounds in the form \eqref{eq1:upbound}.
For this, note that if $r=4$, $R=3$, then  $q^{(r-R)/R}\cdot\sqrt[R]{\ln q}=\sqrt[3]{q\ln q}$.

\subsection{Estimates of $S_w^{\min}$}
 Many important properties of the elliptic quadric are given in \cite[Section 16]{Hirs_PG3q}. In particular, a plane containing at least two points of $\Q$ intersects $\Q$ in a $(q+1)$-arc.
 Through every point of $\Q$ we have one tangent plane and $q(q+1)$ secant planes. We have no planes external to $\Q$.

\begin{lemma}\label{lem2:intersect arcs}
Let $\A_1$ and $\A_2$ be two $(q+1)$-arcs that are intersections of $\Q$ and planes $\pi_1$ and $\pi_2$, respectively.  Then  $\#(\A_1\cap\A_2)=2,1,0$ if $\pi_1\ne\pi_2$ and the intersection line $\pi_1\cap\pi_2$  is a bisecant, a tangent, and an external line to $\Q$, respectively.
\end{lemma}

\begin{proof}
  It easy to see that $A_1 \cap A_2  = \pi_1 \cap \pi_2 \cap \Q$
and the lemma follows immediately.
\end{proof}

\begin{lemma}\label{lem3:Sw}
  Let a  point $P$ of $\PG(3,q)$ be not covered by $\K_{w}$. It is possible $P\in\Q$ and $P\notin\Q$ as well. Then
  \begin{align}\label{eq3:SwP a}
 &  S_w^{\min}\ge \binom{w}{2}\left(q-\binom{w}{2}\right)\text{ if }2\binom{w}{2}-1\le q;\db\\
 &S_w^{\min}\ge\frac{q^2-1}{4} \text{ if }2\binom{w}{2}-1> q\text{ and }q\text{ is odd};\label{eq3:SwP b}\db\\
 &S_w^{\min}\ge\frac{q^2}{4}  \text{ if }2\binom{w}{2}-1> q\text{ and }q\text{ is even}.\label{eq3:SwP c}
  \end{align}
\end{lemma}

\begin{proof}Let
$\pi(T,Q,R)$ be the plane through three \emph{distinct} points $T,Q,R$ of $PG(3,q)$. Let $\pi_w=\{\pi(B_i,B_j,B_k)|1\le i<j<k\le w\}$
 be the multiset of $\binom{w}{3}$ planes through three distinct points of $\K_w$.
For $1\le i<j\le w$, we have $\pi(B_i,B_j,P)\notin\pi_w$ otherwise the point $P$ would be covered by $\K_{w}$.

We consider the $(q+1)$-arc $\A_{i,j}(P)\triangleq\pi(B_i,B_j,P)\cap\Q$; then $\A_{i,j}(P)\cap K_{w}=\{B_i,B_j\}$.
Every point $\A_{i,j}(P)\setminus\{B_i,B_j\}$ provides the inclusion of $P$ into $\N^{\,\cup}_{w+1}$ when one sequentially adds all points $H$ of $\Q\setminus\K_w$ to $\K_w$ to obtain variants $\K_{w+1}=\K_w\cup{H}$. If $P\in\Q$ then one of these points is $P$. So, for a fixed pair $(i,j)$, we have $q-1$ inclusions of $P$  into~$\N^{\,\cup}_{w+1}$.

To estimate $S_w(P)$, we should consider all the $\binom{w}{2}$ pairs $(i,j)$ and take into account all the intersections points of the arcs $\A_{i,j}(P)$ for distinct $(i,j)$ with each other.
Let $\Pi_w(P)\triangleq \{\pi(B_i,B_j,P)\;|\;1\le i<j\le w\}$. If two planes from $\Pi_w(P)$ coincide with each other then we have a plane containing $P$ and three or four points of $\K_w$; it implies that $P$ is covered by $\K_w$, contradiction. So, all the planes of $\Pi_w(P)$ are pairwise distinct.

 The set $\Pi_w(P)$ defines the $\binom{w}{2}$-set $\As_w(P)$ of the $(q-1)$-arcs such that $\As_w(P)\triangleq\{(\pi(B_i,B_j,P)\cap\Q)\setminus\{B_i,B_j\}\;|\;1\le i<j\le w\}$.

 By Lemma \ref{lem2:intersect arcs}, two arcs  from $\As_w(P)$ intersect each other in at most two points. To obtain a lower bound of $S_w(P)$ we can assume that any two arcs of $\As_w(P)$ intersect each other in two points. Moreover, for $P\notin\Q$, we can assume that all intersection points in all pair of the arcs are distinct. This is the worst case. But, for $P\in\Q$, one intersection point is the same for all pairs, it is $P$. Therefore, the lower bound for $P\notin\Q$ is less than the one for $P\in\Q$ and we can consider only $P\notin\Q$.

 Let $P\notin\Q$.
  We take $n$ arcs from $\As_w(P)$. We assume that every arc intersects all $n-1$ other arcs in two points and all the intersection points are distinct; it is the worst case for the value of  $S_w(P)$.  As $(q-1)-2(n-1)$ must be $\ge0$, the considered case is possible if $2n-1\le q$. In all $n$ arcs, the total
number of the intersection points  is $2\binom{n}{2}=n(n-1)$. The total number of distinct points in the union of  $n$ the arcs is $n(q-1)-n(n-1)=n(q-n)$.

If $2\binom{w}{2}-1\le q$ we put $n=\binom{w}{2}$ that proves \eqref{eq3:SwP a}.

 Let $2\binom{w}{2}-1> q$.

Let $q$ be odd. We put $n=(q+1)/2$ that implies $(q-1)-2(n-1)=0$, $n(q-n)=(q^2-1)/4$. So, \eqref{eq3:SwP b} is proved.

If $q$ is even, we put $n=q/2$ that gives $(q-1)-2(n-1)=1$, $n(q-n)=q^2/4$, and proves \eqref{eq3:SwP c}.
\end{proof}

\subsection{Implicit Bound A}\label{subsec:bndA}
By above, in particular by \eqref{eq2:>=aver2} and \eqref{eq3:U3}, the implicit Bound A can be given as follows:
\begin{equation}\label{eq3:BoundA}
 \#\U_3=q^3,~ \#\U_{w+1}= \#\U_w-\left\lceil\frac{ S_w^{\min}\cdot\#\U_w}{q^2+1-w}\right\rceil, ~\#\U_{w^\text{A}_q+1}\le1,~n^\text{A}_{4,q}\triangleq w^\text{A}_q+1.
 \end{equation}
We put $w=3,\#\U_3=q^3$, and then, sequentially increasing $w$ by 1, we calculate $\#\U_{w+1}$, as it is written in \eqref{eq3:BoundA}, until for some $w$ we obtain  $\#\U_{w+1}\le1$; we denote this $w$ as $w^\text{A}_q$; also, let $n^\text{A}_{4,q}\triangleq w^\text{A}_q+1$. We have obtained the implicit upper Bound A
\begin{align}\label{eq3:BoundA2}
s(2,3)=\ell_q(4,3)\le n^\text{A}_{4,q}.
\end{align}
For illustration and comparison, the Bound A in the form $n^\mathrm{A}_{4,q}/\sqrt[3]{q\ln q}$
is shown by the second curve in Figure \ref{fig:bound ShortBnd}, in the region  $7951\le q<10^5$, and the bottom curve in Figure~\ref{fig:LongBnd}, where $10^5< q<5\cdot10^6$. The curves are obtained using $S_w$ from Lemma \ref{lem3:Sw}. By Figure~\ref{fig:bound ShortBnd}, for $7951\le q<10^5$, the Bound A  takes the values $3.38\gtrapprox n^\mathrm{A}_{4,q}/\sqrt[3]{q\ln q}\gtrapprox2.79$. By Figure~\ref{fig:LongBnd}, for $10^5< q<5\cdot10^6$, the Bound A takes the values $2.79\gtrapprox n^\mathrm{A}_{4,q}/\sqrt[3]{q\ln q}\gtrapprox2.63$. Other curves in the figures will be explained below.
\begin{figure}[htbp]
\includegraphics[width=\textwidth]{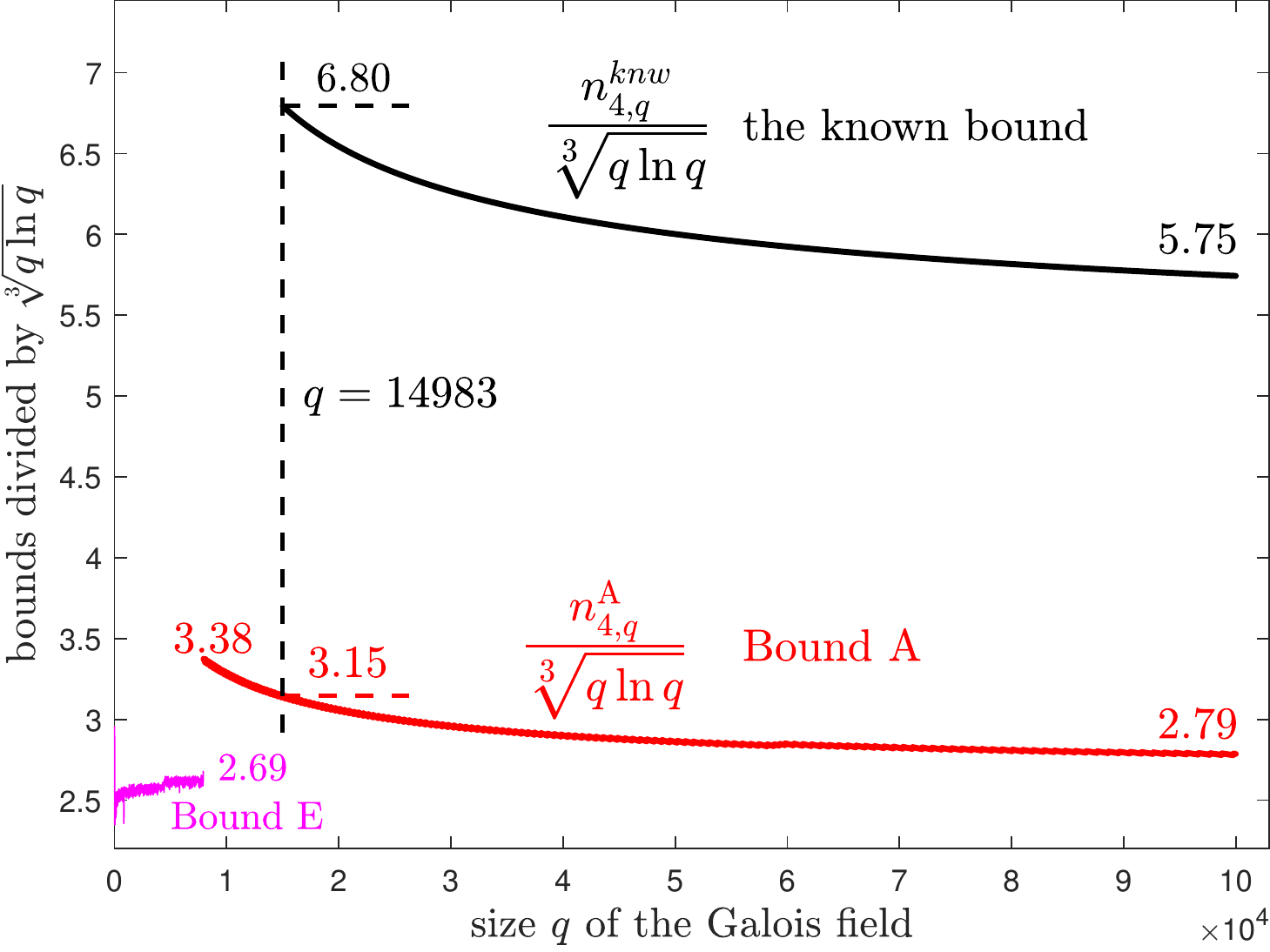}
\caption{Upper bounds on the length function $\ell_q(4,3)$ divided by $\sqrt[3]{q\ln q}$: implicit Bound~A \eqref{eq3:BoundA}--\eqref{eq3:BoundA2} for $7951\le q<10^5$ (\emph{the second curve}), computer Bound E \eqref{eq3:comp_bnd} for $13\le q\le7949$ (\emph{the bottom curve}) vs the known bound \eqref{eq5:lqr3}--\eqref{eq5:lqr3a} for $14983\le q<10^5$ (\emph{the top curve})}
 \label{fig:bound ShortBnd}
\end{figure}

\begin{figure}[htbp]
\includegraphics[width=\textwidth]{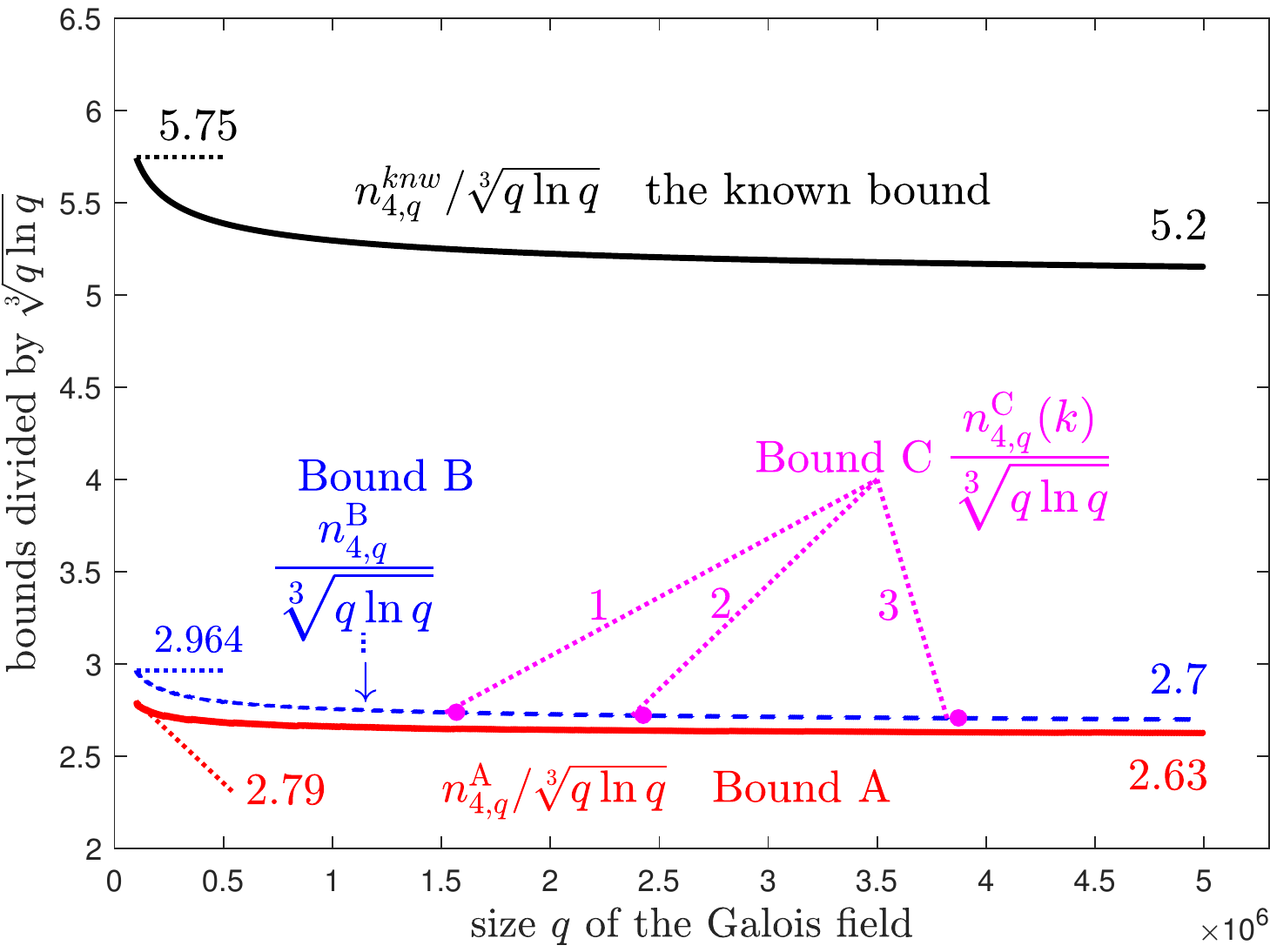}
\caption{Upper bounds on the length function $\ell_q(4,3)$ divided by $\sqrt[3]{q\ln q}$: implicit Bound~A \eqref{eq3:BoundA}--\eqref{eq3:BoundA2} (\emph{the bottom, solid  curve}), implicit Bound B \eqref{eq3:boundB}--\eqref{eq3:bndB_b} (\emph{the second, dashed curve}), and explicit Bound C (\emph{points $1$, $2$, and $3$ correspond to $n^\text{C}_{4,q}(k)/\sqrt[3]{ q\ln q}$ with $k=20.339,\;20$, and $19.7$, by Table $\ref{tab1}$}) vs the known bound \eqref{eq5:lqr3}--\eqref{eq5:lqr3a} (\emph{the top, solid  curve}), $10^5< q<5\cdot10^6$}
 \label{fig:LongBnd}
\end{figure}

\subsection{Implicit Bound B}\label{subsec:BoundB}
By \eqref{eq2:>=aver2} and \eqref{eq3:U3},
\begin{align} \label{eq3:>=aver3}
 &\#\U_{w+1}\le \#\U_w\left(1-\frac{ S_w^{\min}}{q^2+1-w}\right)\le q^3\prod_{j=3}^w\left(1-\frac{ S_j^{\min}}{q^2+1-w}\right).
 \end{align}
Similarly to Section \ref{subsec:bndA}, we put $w=3,\#\U_3=q^3$;  then, sequentially increasing $w$ by 1, we calculate $\#\U_{w+1}$ by \eqref{eq3:>=aver3}, until for some $w$ we obtain   $\#\U_{w+1}\le1$.

Let $q_0$ be a value such that for $q\ge q_0$  in all steps of the considered process we have $2\binom{j}{2}-1\le q$ and use the variant of $S_j^{\min}$ as in \eqref{eq3:SwP a}. Then, for $q\ge q_0$ we have
\begin{align}\label{eq2:Delta2}
&\#\U_{w+1}\le\#\U_{3}\prod_{j=3}^w\left(1-\frac{ \binom{j}{2}\left(q-\binom{j}{2}\right)}{q^2+1-j}\right)=q^3f_q(w),\db\\
&f_q(w)\triangleq\prod_{j=3}^w\left(1-\frac{\binom{j}{2}\left(q-\binom{j}{2}\right)}{q^2+1-j}\right).\label{eq3:fqw}
\end{align}

\begin{lemma}\label{lem3:BoundsB1B2}
 Let $q\ge q_0$. To provide $\#\U_{w+1}\le q^3f_q(w)\le1$ it is sufficient to take $w$ satisfying  the inequality
  \begin{align}
    &(w-1)^3-\frac{0.3w^5}{q}\ge18q\ln q. \label{eq3:boundB}
\end{align}
\end{lemma}

\begin{proof}
By \eqref{eq3:fqw} and the inequality $1-x<\exp(-x)$, we have
\begin{align*}
 &f_q(w)=\prod_{j=3}^w\left(1-\frac{j(j-1)\left(2q-j(j-1)\right)}{4(q^2+1-j)}\right)<\prod_{j=3}^w\exp\left(-\frac{j(j-1)(2q-j(j-1))}{4(q^2+1-j)}\right)\db\\
&< \prod_{j=3}^w\exp\left(-\frac{j(j-1)(2q-j(j-1))}{4q^2}\right)=\exp\left(-\sum_{j=3}^w\frac{j(j-1)(2q-j(j-1))}{4q^2}\right).
 \end{align*}
 After simple transformations, using the results of \cite[Section 1.2.3.1]{HandbookMathForm}, we obtain
\begin{align*}
 & \sum_{j=3}^w j(j-1)=-2+\sum_{j=1}^w j(j-1)=\frac{w^3-w-6}{3}>\frac{(w-1)^3}{3};\db\\
  &\sum_{j=3}^wj^2(j-1)^2=-4+\sum_{j=1}^w j^2(j-1)^2=\frac{w^5}{5}-\frac{w^3}{3}+\frac{2w}{15}-4<\frac{w^5}{5};\db\\
  &f_q(w)<\exp\left(-\frac{w^3-w-6}{6q}+\frac{3w^5-5w^3+2w-60}{60q^2}\right)<\exp\left(-\frac{(w-1)^3}{6q}+\frac{w^5}{20q^2}\right).
\end{align*}
Taking the logarithm of both the parts of the inequality $q^3f_q(w)\le1$, we obtain
\begin{align*}
 &3\ln q-\frac{(w-1)^3}{6q}+\frac{w^5}{20q^2}\le0,
 \end{align*}
 that implies the assertion.
\end{proof}

Let $w^\text{B}_q$ be the smallest integer satisfying the inequality \eqref{eq3:boundB} under the condition  $2\binom{w}{2}-1\le q$. Let $n^\text{B}_{4,q}\triangleq w^\text{B}_q+1$, where ``$+1$'' takes into account that $\#\U_{w+1}\le1$.  We have obtained the implicit upper Bound B:
\begin{align}\label{eq3:bndB_b}
  s_q(2,3)=\ell_q(4,3)\le n^\text{B}_{4,q}.
\end{align}

 Now we will estimate the value of $q_0$.
  If $w^2-w-1\le q $, then $2\binom{w}{2}-1\le q$. So, we can consider $w\le \sqrt{q}$. Let $\delta(q)$  be the difference between the left  and right parts of the inequality  \eqref{eq3:boundB}  under condition that $w=\sqrt{q}$. In other words,
  \begin{align*}
  &\delta(q)\triangleq  (\sqrt{q}-1)^3-0.3q\sqrt{q}-18q\ln q.
  \end{align*}
  Considering the corresponding derivatives, it can be shown that for $q\ge 88274$ we have $\delta(q)>0$. For simplicity of presentation, we can put $q_0=10^5$.

For illustration and comparison, the Bound B in the form $ n^\mathrm{B}_{4,q}/\sqrt[3]{q\ln q}$
is shown by the second,  dashed curve in Figure \ref{fig:LongBnd}, where $q_0=10^5< q<5\cdot10^6$.  By Figure~\ref{fig:LongBnd}, for $10^5< q<5\cdot10^6$, the Bound B takes the values $2.964\gtrapprox n^\mathrm{B}_{4,q}/\sqrt[3]{q\ln q}\gtrapprox2.7$.

\subsection{Explicit Bound C}\label{subsec:explicit}
We will find the solution of the inequality \eqref{eq3:boundB} in the form $w=\lceil \sqrt[3]{kq\ln q}\,\rceil$, where $k>0$  is independent of $q$. For the convenience of research we write $w=\sqrt[3]{k q\ln q}+1$. The inequality  \eqref{eq3:boundB}  takes the form
\begin{align} \label{eq3:boundB2forexplicit}
(k-18)q\ln q-\frac{0.3(\sqrt[3]{kq\ln q}+1)^5}{q}\ge0
\end{align}
that implies $k>18$, $\sqrt[3]{k}>\sqrt[3]{18}\thickapprox2.6207$. Let $\epsilon>0$ be an arbitrary constant independent of $q$. Let $\V$ be a value such that
\begin{equation}
  0.302(\sqrt[3]{(18+\epsilon) q\ln q})^5\ge0.3(\sqrt[3]{(18+\epsilon) q\ln q}+1)^5\text{ if }q>\V.
\end{equation}
Considering the corresponding derivatives, it can be shown that $\V=1516750$.  So, we can
 consider the inequality
\begin{align}\label{eq3:without1}
(k-18)q\ln q-\frac{0.302(\sqrt[3]{kq\ln q}\,)^5}{q}\ge0,~q>\V=1516750,~ k >18.
\end{align}
If \eqref{eq3:without1} holds, then \eqref{eq3:boundB2forexplicit} holds also. We denote
\begin{align}\label{eq3:F(k,q)b}
& F(k,q)\triangleq\left(\frac{k-18}{0.302}\right)^3\frac{1}{k^5}-\frac{\ln^2 q}{q},~q>\V,~ k >18.
\end{align}
The inequality \eqref{eq3:without1}
is equivalent to $F(k,q)\ge0$.
The derivative of $F(k,q)$ with respect to $q$ is
\begin{align}\label{eq3:deriv_q}
 \frac{d}{dq}F(k,q)=\frac{\ln q(\ln q-2)}{q^2}>0,~q>\V.
\end{align}
So, $F(k,q)$ is an increasing function of $q$ for a fixed $k$ and $q>\V$. Therefore, for $q>\V$, the equation $F(k,q)=0$ with a fixed $k$ has only one solution with respect to $q$.
We denote this solution  $\W(k)$. As $\W(k)$ can be non-integer, below we use $\lceil\W(k)\rceil$.
By above, $\lceil\W(k)\rceil$ is the smallest integer $q$ satisfying \eqref{eq3:without1} for the fixed $k$. Thus, for $q\ge\lceil\W(k)\rceil\ge\W(k)$, the inequality \eqref{eq3:without1} holds under the condition $\lceil\W(k)\rceil>\V=1516750$.

By \eqref{eq3:F(k,q)b}, the equation $F(k,q)=0$ can be written in the form
\begin{align}\label{eq3:Fqk=0}
 \frac{\ln^2 q}{q}=\left(\frac{k-18}{0.302}\right)^3\frac{1}{k^5},~k>18.
\end{align}
Note that this equation is connected with Lambert W function, see e.g. \cite{Lambert}.

By above we have the theorem.

\begin{theorem}\label{th3:BoundC} \textbf{\emph{(explicit Bound C)}}
  Let $\W(k)$ be the solution with respect to  $q$ of the equation~\eqref{eq3:Fqk=0}. Let $k>18$ be such that $\W(k)>\V=1516750$.
  Let $w^\mathrm{C}_q(k)\triangleq \sqrt[3]{k q\ln q}+1$, $n^\mathrm{C}_{4,q}(k)\triangleq w^\mathrm{C}_q(k)+1=\sqrt[3]{k q\ln q}+2$.  Then
  \begin{align}\label{eq3:BoundC}
 s_q(2,3)=\ell_q(4,3)\le n^\mathrm{C}_{4,q}(k)=\sqrt[3]{k q\ln q}+2 \text{ for } q\ge\lceil\W(k)\rceil.
  \end{align}
  We  call $n^\mathrm{C}_{4,q}(k)$ explicit Bound  C.
\end{theorem}

\begin{example}
  Some values of $\lceil\W(k)\rceil$ are given in Table~\ref{tab1}.  Also, in the 3-rd column of the table the values of $n^\text{C}_{4,q}(k)/\sqrt[3]{ q\ln q}$ are written such that
\begin{align}\label{eq3:cnew}
 \frac{n^\text{C}_{4,q}(k)}{\sqrt[3]{ q\ln q}}=\sqrt[3]{k}+\frac{2}{\sqrt[3]{ q\ln q}},~q=\lceil\W(k)\rceil.
\end{align}
They are useful, in particular, for comparison with the known results, see Section \ref{sec:compare}, the 4-th and 5-th columns of Table \ref{tab1}, and Figure~\ref{fig:LongBnd}, where the values of $n^\text{C}_{4,q}(k)/\sqrt[3]{ q\ln q}$ with $k=20.339,20,19.7$ are presented.
\begin{table}[h]
\begin{center}
\caption{Values of $\lceil\W(k)\rceil$ for $18.0001\le k\le 20.340$  and values of
$n^\text{C}_{4,q}(k)/\sqrt[3]{ q\ln q}$, $n^{knw}_{4,q}/\sqrt[3]{q\ln q}$ (the known bound), $n^{knw}_{4,q}/n^\text{C}_{4,q}(k)$ for $q=\lceil\W(k)\rceil$, $18.0001\le k\le 20.339$. ($\V=1516750$)}
\begin{tabular}{l|r|c|c|c}
  \hline
\label{tab1}
  $k$ & $\lceil\W(k)\rceil$ \phantom{\hspace{2cm}}&$\frac{n^\text{C}_{4,q}(k)}{\sqrt[3]{ q\ln q}}$&
  $\frac{n^{knw}_{4,q}}{\sqrt[3]{q\ln q}}$&$\frac{n^{knw}_{4,q}}{n^\text{C}_{4,q}(k)}$ \\\hline
  20.340&$1515738\thickapprox1.516\cdot10^{6}<\V$& &&\\\hline
  20.339&$1517567\thickapprox1.518\cdot10^{6}>\V$& 2.7368&5.2500& 1.9183\\
  20.335&$1524915\thickapprox1.525\cdot10^{6}>\V$&2.7367&5.2495& 1.9182\\
  20&$2374364\thickapprox2.374\cdot10^{6}>\V$&2.7205&5.2087& 1.9146\\
  19.7&$ 3820987\thickapprox3.821\cdot10^{6}>\V$&2.7059&5.1716&1.9112\\
  19&$19178705\thickapprox1.918\cdot10^{7}>\V$&2.6713&5.0828&1.9027\\
  18.5&$171670620\thickapprox1.717\cdot10^{8}>\V$&2.6461&5.0180&1.8963\\
  18.1&$30640000001\thickapprox3.064\cdot10^{10}>\V$& 2.6258&4.9659&1.8912\\
  18.05&$294427001643\thickapprox2.944\cdot10^{11}>\V$&2.6233&4.9593&1.8905\\
  18.01&$52060446118120\thickapprox5.206\cdot10^{13}>\V$&2.6212&4.9542&1.8900\\
  18.001&$\thickapprox7.880\cdot 10^{16}>\V$&2.6208&4.9530&1.8899\\
  18.0001&$\thickapprox1.109\cdot 10^{20}>\V$&2.6207&4.9529&1.8899\\  \hline
\end{tabular}
\end{center}
\end{table}
\end{example}

 The derivative of $F(k,q)$ with respect to $k$ is
\begin{align}\label{eq3:deriv_k}
 &\frac{d}{dk}F(k,q)=\frac{(k-18)^2}{0.302^3 k^5}\left(\frac{90}{k}-2\right)>0, ~18 < k\le 20.339 .
\end{align}

By \eqref{eq3:deriv_k}, $F(k, q)$ is an increasing function of $k$ for a fixed $q$.
Let $F(k',q')=0$ and let $k''>k'$. Then  $F(k'',q')>0$. As, by \eqref{eq3:deriv_q}, $F(k,q)$ is an increasing function of $q$ for a fixed $k$,
there exists $q''<q'$ such that $F(k'',q'')=0$.
So, $\W(k)$  is a decreasing function of $k$ and, therefore, $\lceil W(k)\rceil$ is a non-increasing function of $k$.

As  $\lceil\W(k)\rceil$ should be $>\V$, by Table \ref{tab1} we are interested for $18.0001 \le k\le 20.339 $.

Finally, we have
\begin{align}\label{eq3:BoundCinf}
 \lim_{q\rightarrow\infty}\frac{n^\text{C}_{4,q}(k)}{\sqrt[3]{ q\ln q}}=\sqrt[3]{k},~20.339\ge k>18,~k\text{ fixed}.
\end{align}
\subsection{Asymptotic Bound D}\label{subsec:BoundD}

The inequality \eqref{eq3:without1} can be written in the form
\begin{align}\label{eq3:without1D}
q\ln q\left((k-18)-0.302k\sqrt[3]{\frac{k^2\ln^2 q}{q}}\right)\ge0,~q>\V,~18< k\le 20.339.
\end{align}
We have
\begin{align*}
 \lim_{q\rightarrow\infty}0.302k\sqrt[3]{\frac{k^2\ln^2 q}{q}}=0,~18< k\le 20.339.
\end{align*}
Therefore, for $q$ large enough, the inequality \eqref{eq3:without1D}  holds;
together with it, the inequalities \eqref{eq3:without1} and \eqref{eq3:boundB2forexplicit} hold also. Thus (see the beginning
 of Section~\ref{subsec:explicit}),
for $w=\sqrt[3]{k q\ln q}+1$, $18< k\le 20.339$, the inequality \eqref{eq3:boundB} holds if $q$ is large enough. We take $k=18+\varepsilon$ where $\varepsilon>0$ is an arbitrary small constant independent of~$q$. Let $w^{\text{D}}_q\triangleq\sqrt[3]{(18+\varepsilon) q\ln q}+1$. Then
  \begin{align}\label{eq3:BoundD}
 s_q(2,3)=\ell_q(4,3)\le n^\mathrm{D}_{4,q}\triangleq w^\mathrm{D}_q+1=\sqrt[3]{(18+\varepsilon) q\ln q}+2 \text{ for } q\text{ large enough}.
  \end{align}
  We  call $n^\mathrm{D}_{4,q}$ asymptotic Bound D. We have
  \begin{align}\label{eq3:BoundDinf}
 \lim_{q\rightarrow\infty,\,\varepsilon\rightarrow0}\frac{n^\text{D}_{4,q}}{\sqrt[3]{ q\ln q}}=\sqrt[3]{18}\thickapprox2.6207.
\end{align}

 \subsection{Bound E for relatively small $q$}\label{subsecComput}
 In \cite{BDMP-arXivR3_2018,DMP-R=3Redun2019,DMP-ICCSA2020,DMP-AMC2021,DavOst-DESI2010}, see also the references therein, small 2-saturating sets in $\PG(3,q)$ for the region $13\le q\le7577$ are obtained by computer search using the so-called ``algorithms with the fixed order of points (FOP)'' and ``randomized greedy algorithms''. These algorithms are described in detail in \cite{BDMP-arXivR3_2018,DMP-R=3Redun2019}.

In this paper, we continue the computer search and obtain new small 2-saturating sets  in $\PG(3,q)$ for the region $7577<q\le7949$.

We denote by $\overline{s}_q(3,2)$ the size of the smallest known 2-saturating sets in $\PG(3,q)$. The sets obtained in \cite{BDMP-arXivR3_2018,DMP-R=3Redun2019,DMP-ICCSA2020,DMP-AMC2021,DavOst-DESI2010} and in this paper (one set of \cite{Sonino} is used also)  provide the following theorem.

\begin{theorem}
In the projective space $\PG(3,q)$, for  the smallest size $s_q(3,2)$ of a $2$-saturating set the following upper bound holds:
\begin{equation}\label{eq3:comp_bnd}
s_q(3,2)\le\overline{s}_q(3,2)\le n^\mathrm{E}_{4,q}\triangleq c^{\mathrm{E}}_{4}\sqrt[3]{q\ln q},~
~c^{\mathrm{E}}_{4}= \left\{\begin{array}{@{}l@{}}
  2.61 \text{ if }  13\le q\le4373 \smallskip\\
  2.65\text{ if } 4373< q\le7723 \smallskip\\
  2.69\text{ if } 7723<q\le 7949
\end{array}
  \right..
\end{equation}
\end{theorem}
For  illustration and  comparison, the Bound E in the form $n^\mathrm{E}_{4,q}/\sqrt[3]{q\ln q}$
is shown by the bottom curve in Figure \ref{fig:bound ShortBnd}, in the region  $13\le q<7949$.
%It is interesting that for $18.5\ge k>18$, the theoretical results $c^{\mathrm{C}}_{4,q}(k)$ of Section \ref{subsec:explicit} is better than $c^{\mathrm{E}}_{4,q}$ of \eqref{eq3:comp_bnd}. However, these values of $c^{\mathrm{C}}_{4,q}(k)$ hold for $q$ essentially greater than in \eqref{eq3:comp_bnd}.

\section{Codes with growing codimension $r=3t+1$, $t\ge1$}\label{sec:r=3t+1}

For upper bounds on the length function $\ell_q(r,3)$, $r=3t+1\ge7$, an important tool is given by the inductive lifting
constructions
of \cite{DMP-ICCSA2020,DavOst-DESI2010} which provide Proposition~\ref{prop7_induct_r=3t+1}.  These constructions are variants of $q^m$-concatenating constructions \cite{Dav90PIT,Dav95,DGMP-AMC,DavOst-IEEE2001,DavOst-DESI2010,DMP-AMC2021,DMP-IEEE2004,DMP-ICCSA2020}. We denote
\begin{align*}
 \Delta(r,q)\triangleq3\left\lfloor q^{(r-7)/3}\right\rfloor+2\left\lfloor q^{(r-10)/3}\right\rfloor+\delta_{r,13},~\delta_{i,j}\text{ is the Kronecker symbol}.
\end{align*}
\begin{proposition}\label{prop7_induct_r=3t+1}
 \emph{ \cite[Proposition 21]{DMP-ICCSA2020},\cite[Theorem 14]{DavOst-DESI2010}}
  Let an $ [n_{0},n_{0}-4]_{q}3$ code with $n_0<q$ exist. Then there is an infinite family of
$[n,n-r]_{q}3$ codes with parameters
\begin{align}\label{eq4:lift}
&  n=n_{0}q^{(r-4)/3}+\Delta(r,q),~ r=3t+1\geq 4,~t\ge1.
\end{align}
\end{proposition}

\begin{corollary}\label{cor4}
  Let $n_0=n^\mathrm{\bullet}_{4,q}\in \{n^\mathrm{A}_{4,q},n^\mathrm{B}_{4,q},n^\mathrm{D}_{4,q},n^\mathrm{E}_{4,q}\}$ or $n_0= n^\mathrm{C}_{4,q}(k)$, where  $n^\mathrm{A}_{4,q}$, $n^\mathrm{B}_{4,q}$, $n^\mathrm{D}_{4,q}$, $n^\mathrm{E}_{4,q}$ ,and $n^\mathrm{C}_{4,q}(k)$ are given in Section \ref{sec:2sat}. Let the region of $q$ for $n^\mathrm{\bullet}_{4,q}$ and $n^\mathrm{C}_{4,q}(k)$ be as in Section \ref{sec:2sat}. Then, for the same region of $q$, there is an infinite family of $[n^\bullet_{r,q},n^\bullet_{r,q}-r]_{q}3$ or $[n^\mathrm{C}_{r,q}(k),n^\mathrm{C}_{r,q}(k)-r]_q3$ codes with $r=3t+1\geq 4$, $t\ge1$, and length
  \begin{align}\label{eq4:lift-b}
&  n^\bullet_{r,q}=n^\bullet_{4,q}q^{(r-4)/3}+\Delta(r,q),
~n^\mathrm{\bullet}_{4,q}\in \{n^\mathrm{A}_{4,q},n^\mathrm{B}_{4,q},n^\mathrm{E}_{4,q}\};\db\\
&  n^\mathrm{C}_{r,q}(k)=\sqrt[3]{k}\cdot q^{(r-3)/3}\cdot\sqrt[3]{\ln q}+2q^{(r-4)/3}+\Delta(r,q),~18 <k\le20.339,~  q\ge\lceil\W(k)\rceil;\label{eq4:liftBndC}\db\\
&n^\mathrm{D}_{r,q}(k)=\sqrt[3]{18+\varepsilon}\cdot q^{(r-3)/3}\cdot\sqrt[3]{\ln q}+2q^{(r-4)/3}\label{eq4:liftBndD}
+\Delta(r,q),~ q\text{ large enough}.
\end{align}
\end{corollary}
\begin{proof}
  By Section \ref{sec:2sat}, for all $n^\mathrm{A}_{4,q},n^\mathrm{B}_{4,q},n^\mathrm{C}_{4,q}(k),n^\mathrm{D}_{4,q},n^\mathrm{E}_{4,q}$ we have $n_0=n^\mathrm{\bullet}_{4,q}<q$ and $n_0= n^\mathrm{C}_{4,q}(k)<q$.
\end{proof}
By \eqref{eq3:BoundCinf}, \eqref{eq3:BoundDinf}, \eqref{eq4:liftBndC}, \eqref{eq4:liftBndD}, for $r=3t+1\geq 4,~t\ge2$, we have
\begin{align}
&\lim_{q\rightarrow\infty}\frac{n^\text{C}_{r,q}(k)}{q^{(r-3)/3}\cdot\sqrt[3]{\ln q}}=
\lim_{q\rightarrow\infty}\frac{n^\text{C}_{4,q}(k)}{q^{(4-3)/3}\cdot\sqrt[3]{\ln q}}=\sqrt[3]{k},~18< k \le 20.339,~k\text{ fixed};\dbn\\
& \lim_{q\rightarrow\infty,\,\varepsilon\rightarrow0}\frac{n^\text{D}_{r,q}}{q^{(r-3)/3}\cdot\sqrt[3]{\ln q}}=\lim_{q\rightarrow\infty,\,\varepsilon\rightarrow0}\frac{n^\text{D}_{4,q}}{q^{(4-3)/3}\cdot\sqrt[3]{\ln q}}=\sqrt[3]{18}\thickapprox2.6207.\label{eq4:lqr3D}
\end{align}

\section{Comparison with the known results}\label{sec:compare}
As far as the authors know, the best known bounds on $ \ell_q(3t+1,3)$ are given in \cite{DMP-AMC2021} where for $R=3$, $r=3t+1$, $t\ge1$, the following results are obtained:
\begin{align}
& \ell_q(4,3)\le\Omega_{\lambda,3}(q)\sqrt[3]{q\ln q}+6,~
\ell_q(r,3)\le(\Omega_{\lambda,3}(q)\sqrt[3]{q\ln q}+6)q^{\frac{r-4}{3}}+3\theta_{t-1,q},\label{eq5:lqr3}\db\\
&\Omega_{\lambda,3}(q)\triangleq\lambda+\frac{36}{\beta_{\lambda,3}^2(q)
\left(2-\frac{1}{q}-\Upsilon_{\lambda,3}(q)\right)},~\Upsilon_{\lambda,3}(q)\triangleq\frac{\lambda^2}{2}\sqrt[3]{\frac{\ln^2 q}{q}},~\beta_{\lambda,3}(q)\triangleq\lambda-\frac{2}{\sqrt[3]{q\ln q}}.\notag
\end{align}
Here $\lambda > 0$ is a positive constant independent of $q$, its value can be assigned
arbitrarily. The bounds~\eqref{eq5:lqr3} hold if $q>\lceil y\rceil$ where $y$ is a solution of the equation $\Upsilon_{\lambda,3}(y)=1$ under the condition $y > e^2$ \cite[(3.7), Remark 6.4]{DMP-AMC2021}.

In \cite[Section~6, Table 1]{DMP-AMC2021}, it is shown that $\lambda=\sqrt[3]{36}$ minimizes
 $\Omega_{\lambda,3}$;  for this $\lambda$, the bounds~\eqref{eq5:lqr3} hold if $q\ge14983$. We denote
 \begin{align}\label{eq5:lqr3a}
 &n^{knw}_{4,q}\triangleq\Omega_{\sqrt[3]{36},3}(q)\sqrt[3]{q\ln q}+6,
 ~n^{knw}_{r,q}\triangleq n^{knw}_{4,q}\cdot q^{\frac{r-4}{3}}+3\theta_{t-1,q},~r=3t+1,~t\ge1,
 \end{align}
 where ``$knw$'' notes the known results. For  illustration and  comparison, the known bound $n^{knw}_{4,q}$ in the form $n^{knw}_{4,q}/\sqrt[3]{q\ln q}$
is shown by the top curve in Figure \ref{fig:bound ShortBnd}, for $14983\le q<10^5$, and Figure \ref{fig:LongBnd}, for $10^5<q<5\cdot 10^6$. By Figure~\ref{fig:bound ShortBnd}, for $14983\le q<10^5$, the known bound  takes values $6.80\gtrapprox n^{knw}_{4,q}/\sqrt[3]{q\ln q}\gtrapprox5.75$. By Figure~\ref{fig:LongBnd}, for $10^5< q<5\cdot10^6$, the known bound takes values $5.75\gtrapprox n^{knw}_{4,q}/\sqrt[3]{q\ln q}\gtrapprox5.2$.

The ratio $n^{knw}_{4,q}/n^\mathrm{A}_{4,q}$ of the known upper bound $n^{knw}_{4,q}$ \eqref{eq5:lqr3}--\eqref{eq5:lqr3a} on the length function $\ell_q(4,3)$ and the new one  $n^\mathrm{A}_{4,q}$ in the region $14983\le q<5\cdot10^6$ is shown in Figure \ref{fig:ratio}, where the ratio lies in $\thickapprox2.167\ldots\thickapprox1.96$. The ratio $n^{knw}_{4,q}/n^\mathrm{C}_{4,q}(k)$ of the known bound  and the new one  $n^\mathrm{C}_{4,q}(k)$ in the region $1517567<q$ is shown in Table  \ref{tab1} where it lies in $\thickapprox1.918\ldots\thickapprox1.89$.
\begin{figure}[h]%[htbp]
 \includegraphics[width=\textwidth]{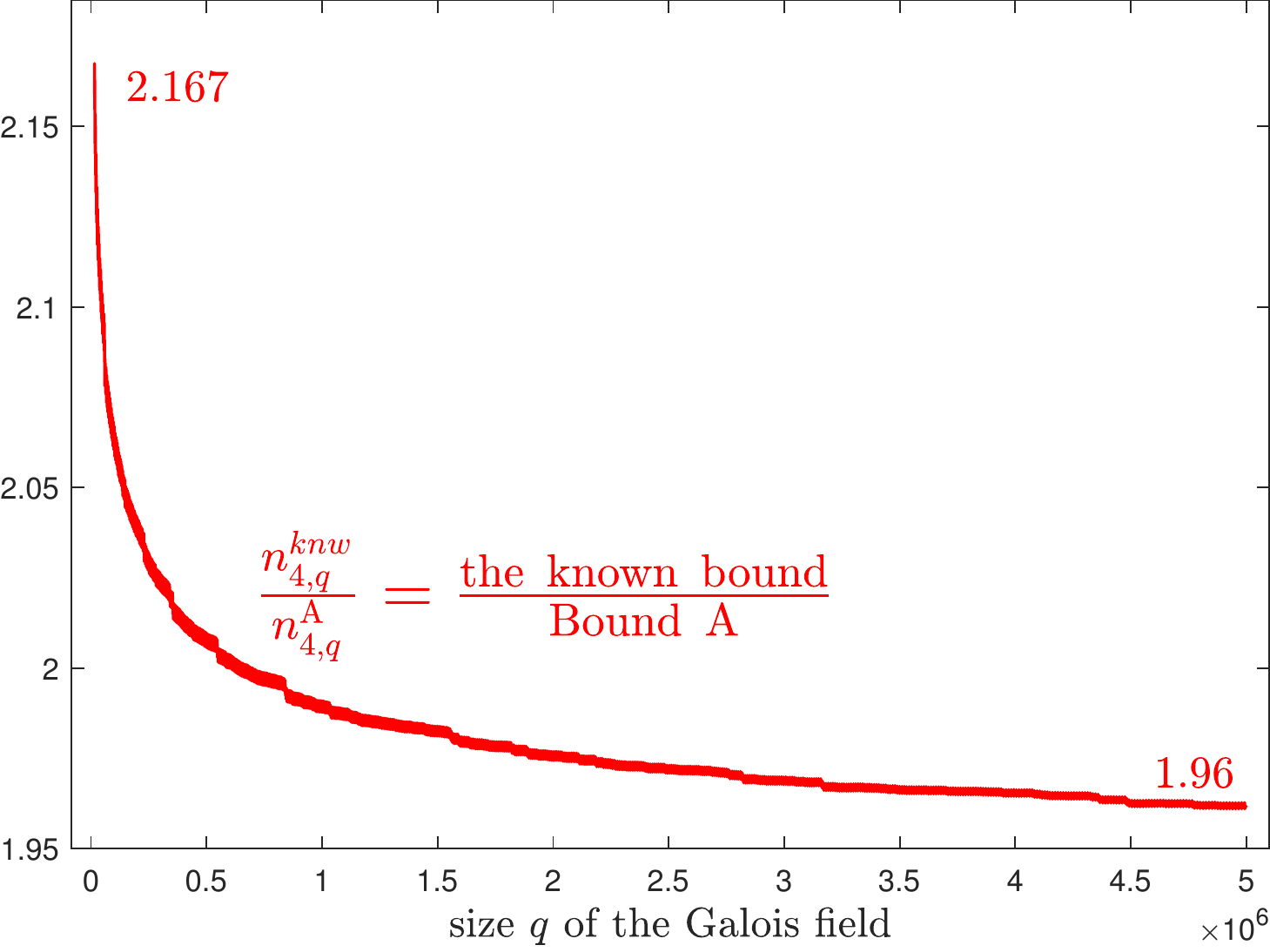}
\caption{The ratio $n^{knw}_{4,q}/n^\mathrm{A}_{4,q}$ of the known \cite{DMP-AMC2021} upper bound $n^{knw}_{4,q}$ \eqref{eq5:lqr3}--\eqref{eq5:lqr3a} on the length function
$\ell_q(4,3)$ and the new one  $n^\mathrm{A}_{4,q}$  \eqref{eq3:BoundA}--\eqref{eq3:BoundA2}, $14983\le q<5\cdot10^6$ }
 \label{fig:ratio}
\end{figure}

 By Table  \ref{tab1}, \eqref{eq4:lqr3D}, \eqref{eq5:lqr3}, for asymptotic estimates we have
\begin{align}\label{eq5:lqr3c}
&\lim_{q\rightarrow\infty}\frac{n^{knw}_{4,q}}{\sqrt[3]{q\ln q}}=
\lim_{q\rightarrow\infty}\frac{n^{knw}_{r,q}}{q^{(r-3)/3}\cdot\sqrt[3]{\ln q}}=\frac{3}{2}\sqrt[3]{36}\thickapprox 4.953;\db\\
&\lim_{q\rightarrow\infty,\,\varepsilon\rightarrow0}\frac{n^{knw}_{4,q}}{n^\text{D}_{4,q}}\thickapprox1.5\sqrt[3]{2}\thickapprox1.8899.\label{eq5:inf ratio}
\end{align}

By Figures \ref{fig:bound ShortBnd}--\ref{fig:ratio}, Table \ref{tab1}, and \eqref{eq5:inf ratio},
 one sees that the new bounds are essentially better than the known ones.

 \end{document}